\documentclass[11pt,reqno]{amsart}
\usepackage{amsmath,amsthm,amscd,amsfonts,amssymb,color}
\usepackage{cite}
\usepackage[mathscr]{eucal}

\usepackage[bookmarksnumbered,colorlinks,plainpages]{hyperref}
\newtheorem{theorem}{Theorem}[section]

\newtheorem*{Acknowledgement}{\textnormal{\textbf{Acknowledgement}}}

\newtheorem{corollary}[theorem]{Corollary}
\theoremstyle{definition}
\newtheorem{definition}[theorem]{Definition}
\newtheorem{example}[theorem]{Example}
\newtheorem{Open Prob}[theorem]{Open Problem}
\theoremstyle{remark}
\newtheorem{remark}[theorem]{Remark}
\numberwithin{equation}{section}
\def\DJ{\leavevmode\setbox0=\hbox{D}\kern0pt\rlap{\kern.04em\raise.188\ht0\hbox{-}}D}
\begin{document}
\title[TOPOLOGICAL DEVELOPMENTS OF $\mathcal{F}$-METRIC SPACES]{TOPOLOGICAL DEVELOPMENTS OF $\mathcal{F}$-METRIC SPACES}
\author[A.\ Bera, L.K.\ Dey, H.\ Garai, A.\ Chanda]
{Ashis Bera$^{1}$,  Lakshmi Kanta Dey$^{2}$, Hiranmoy Garai$^{3}$, Ankush Chanda$^{4}$}

\address{{$^{1}$\,} Ashis Bera,
                    Department of Mathematics,
                    National Institute of Technology
                    Durgapur, India.}
                    \email{beraashis.math@gmail.com}
\address{{$^{2}$\,} Lakshmi Kanta Dey,
                    Department of Mathematics,
                    National Institute of Technology
                    Durgapur, India.}
                    \email{lakshmikdey@yahoo.co.in}
\address{{$^{3}$\,} Hiranmoy Garai,
                    Department of Mathematics,
                    National Institute of Technology
                    Durgapur, India.}
                    \email{hiran.garai24@gmail.com}          
                    
\address{{$^{4}$\,} Ankush Chanda,
                    Department of Mathematics,
                    National Institute of Technology
                    Durgapur, India.}
                    \email{ankushchanda8@gmail.com}


\begin{abstract}
In this manuscript, we claim that the newly introduced $\mathcal{F}$-metric spaces are Hausdorff and also first countable. Moreover, we assert that every separable $\mathcal{F}$-metric space is second countable. Additionally, we acquire some interesting fixed point results concerning altering distance functions for  contractive-type mappings and Kannan-type contractive mappings in this exciting context. However, most of the findings are well-furnished by several non-trivial numerical examples. Finally, we raise an open problem regarding the metrizability of such kind of spaces.
\end{abstract}
 
\maketitle

\setcounter{page}{1}

\centerline{}

\centerline{}

\section{\bf Introduction}
\baselineskip .55 cm
As a generalization of Euclidean geometry and a common setting for continuous functions, topology of metric spaces is one of the most fascinating and instrumental branches of research in contemporary mathematics. This fact has prompted many mathematicians to deal with the topology induced by a metric on a non-empty set in plenty of research articles. Therefore the topology of $b$-metric spaces, dislocated metric spaces and several other abstract spaces are thoroughly investigated and also improved by several authors (see \cite{D3,M2,TFAS,NA1,HS} and references therein). 

On the other hand, metric fixed point theory appears as one of the most salient means to work out various research ventures in non-linear functional analysis and a variety of other fields in science and technology. It all emerged with the illustrious Banach contraction principal, due to Banach \cite{B1}, in 1922 and subsequently, plenty of results appeared which complement, extend and obviously improve the pioneer one \cite{CMDK,DD,GRRS,HS,R5,S3,G,C,JS,GDMR}.

Right through the years, mathematicians got involved with improving the underlying metric structure of the previous result and in a very recent article, Jleli and Samet \cite{JS1} proposed another interesting framework to work with. The authors made use of a certain class of functions to coin the notion of an $\mathcal{F}$-metric space. We begin with the collection of the auxiliary functions.

Let $\mathcal{F}$ be the set of functions $f:[0,\infty)\rightarrow \mathbb{R}$ satisfying the following conditions:

($\mathcal{F}_1$) $f$ is non-decreasing, i.e., $ 0<s<t\Rightarrow f(s)\leq f(t)$.

($\mathcal{F}_2$) For every sequence $\{t_n\}\subseteq (0,+\infty)$, we have
$$\lim_{x\to +\infty}t_n=0 \Longleftrightarrow \lim_{x\to +\infty}f(t_n)=-\infty.$$

Utilizing such functions the authors generalized the concept of usual metric spaces and originated the notion of $\mathcal{F}$-metric spaces as follows:
\begin{definition} \cite{JS1} \label{D1}
Let $X$ be a non-empty set, and let $d:X\times X\rightarrow [0,\infty)$ be a given mapping. Suppose that there exists $(f,\alpha)\in \mathcal{F}\times [0,\infty)$ such that
\begin{enumerate}
\item[(D1)] $(x,y)\in X \times X,~~ d(x,y)=0\Longleftrightarrow x=y$.
\item[(D2)] $d(x,y)=d(y,x),$ for all $(x,y)\in X \times X$.
\item[(D3)] For every $(x,y)\in X\times X$, for each $N\in \mathbb{N}, N\geq2$, and for every $(u_i)_{i=1}^{N}\subseteq X $ with $(u_1,u_N)=(x,y)$, we have
$$d(x,y)>0 \Longrightarrow f(d(x,y))\leq f\Big(\sum_{i=1}^{N-1} d(u_i, u_{i+1})\Big)+ \alpha.$$
\end{enumerate}
Then $d$ is said to be an $\mathcal{F}$-metric on $X$, and the pair $(X,d)$ is said to be an $\mathcal{F}$-metric space.
\end{definition}
We observe that any metric on $X$ is an $\mathcal{F}$-metric, but the converse is not true, which is given in \cite[\, Example 2.1]{JS1}. Besides, the succeeding one is an example of rectangular $b$-metric which is not an $\mathcal{F}$-metric. This implies that the set of rectangular $b$-metric spaces is not contained in the collection of $\mathcal{F}$-metric spaces.
\begin{example}
Let $X=\mathbb{N}$ and define a map $d:X\times X\rightarrow [0,\infty)$ such that $d(x,y)=d(y,x)$ for all $x,y\in X$ and
\[d(x,y)=
\begin{cases}
0,\; x=y;\\
15, x=1, y=20;\\
1, \;x\in\{1,20\}, \; \text{and}\; y\in\{25\};\\
2, \;x\in\{1,20,25\}, \; \text{and}\; y\in\{30\};\\
3, \;x\in\{1,20,30\}, y\in\{35\}; \\
\frac{3}{n^2}, \;x\; \text{or}\; y\notin\{1,20,25,30\} \; \text{and} \;x\neq y.
\end{cases}
\]
Then it can be easily verified that $(X,d)$ is a rectangular $b$-metric space. Suppose that there exists $(f,\alpha)\in \mathcal{F}\times [0,\infty)$ such that $d$ satisfies (D3). Let $n\in \mathbb{N}$ and $u_i=\frac{i}{(n+1)^2}$, $i=1,2,\cdots,n,~~\mbox{with}~~ i\neq k(n+1)^2$. Now, by (D3) we get
\begin{align*}
f(d(x,y))\leq f \Big(\sum_{i=1}^{N-1}d(u_i,u_{i+1})\Big)+\alpha \; \text{with} \; (u_1,u_N)=(x,y) \; \text{and} \; \alpha>0.
\end{align*}
Therefore for some $N\in \mathbb{N}$,
\begin{align*}
f(d(1,20))\leq f(d(1,u_1)+d(u_1,u_2)+...+d(u_{N-1},20))+ \alpha,
\end{align*}
which implies
\begin{align*}
f(15)&\leq f\Big(\frac{3}{n^2}+\frac{3}{n^2}+...+\frac{3}{n^2}\Big)+ \alpha\\
&\leq f\Big(\frac{3}{n}\Big)+ \alpha\rightarrow -\infty 
\end{align*}
as $\displaystyle{\lim_{n\rightarrow \infty}}f\Big(\frac{3}{n}\Big)=-\infty$, which is a contradiction.
\end{example}
Further, the notions of completeness, convergence and Cauchy sequences in this framework along with some other terminologies can be found in \cite{JS1}.

In this literature, we assert a couple of topological observations concerning the newly introduced $\mathcal{F}$-metric spaces. In fact, being a vast generalization of usual metric spaces, such kind of spaces still hold some beautiful topological properties like Hausdorff and also first countability. However, we also confirm that whenever an $\mathcal{F}$-metric space is separable, it is actually a  second countable one. On the other hand, we also enquire into a few exciting fixed point results involving altering distance functions in the later half of this article. Finally, we pose an interesting open problem and construct several non-trivial examples to validate the obtained theorems.
\section{\bf Results on topology of \texorpdfstring{$\mathcal{F}$}{}-metric spaces}
In this section, we deal with the topological developments of $\mathcal{F}$-metric spaces which is equipped with the $\mathcal{F}$-metric topology $\tau_{\mathcal{F}}$. First of all, we attest that such metric spaces are Hausdorff.
\begin{theorem}
Every $\mathcal{F}$-metric space $(X,d)$ is Hausdorff.
\end{theorem}
\begin{proof}
Let $(X,d)$ be an $\mathcal{F}$-metric space, so there exists $(f,\alpha)\in \mathcal{F}\times [0,\infty)$ satisfying the conditions $(D1$-$D3)$ of Definition \ref{D1}. Let $x,y$ be two arbitrary points in $X$ with $x\neq y$, and take $a_n=\frac{d(x,y)}{n}$. Then $\{a_n\}$ is a sequence in $(0,\infty)$ and $a_n\rightarrow 0 $ as $n\rightarrow \infty$. So by $\mathcal{F}_2$, we have, $$f(a_n) \to -\infty~~\mbox{as}~~ n \to \infty.$$
We claim that $B(x,\frac{a_n}{2}) \cap B(y,\frac{a_n}{2}) = \emptyset$ for some $n \in \mathbb{N}$. Suppose to the contrary that $B(x,\frac{a_n}{2}) \cap B(y,\frac{a_n}{2}) \neq \emptyset$. Then we can find a sequence $\{z_n\}$ in $X$ such that $z_n \in B(x,\frac{a_n}{2}) \cap B(y,\frac{a_n}{2})$ for all $n \in \mathbb{N}$. Therefore, $$d(x,z_n)<\frac{a_n}{2}\;\mbox{and} \;d(y,z_n)<\frac{a_n}{2}.$$
Since, $d(x,y)>0$, so using (D3), we get,
\begin{align*}
 f(d(x,y))&\leq f(d(x,z_n)+d(z_n,y))+\alpha\\
&\leq f\Big(\frac{a_n}{2}+\frac{a_n}{2}\Big)+\alpha\\
&= f(a_n) +\alpha.
\end{align*}
Taking limit as $n \to \infty$ in both sides of above equation we get $$f(d(x,y))\leq -\infty,$$ which is a contradiction.


Therefore, $B(x,\frac{a_n}{2}) \cap B(y,\frac{a_n}{2})=\emptyset$ for some $n \in \mathbb{N}$. Hence we are done.
\end{proof}
\begin{remark}
It is worthy to mention that the Hausdorff property is a sufficient condition to claim the uniqueness of a limit for a  convergent sequence. Therefore, this property holds good for every $\mathcal{F}$-metric space also.
\end{remark}
Now we study the other separation axiom in the following result.
\begin{theorem}
Every $\mathcal{F}$-metric space $(X,d)$ is first countable.
\end{theorem}
\begin{proof}
Let $x \in X$ be arbitrary and for each $x\in X$, the family $$\beta=\{B\Big(x,\frac{1}{n}\Big):n\in \mathbb{N}\}$$ is a countable set of open neighborhoods of $x$. Let $U$ be an open neighborhood of $x$. Then by the definition of an $\mathcal{F}$-open set, $B(x,r)\subseteq U$, where $r>0$. By the Archimedean property, there exists $n \in \mathbb{N}$ such that $0<\frac{1}{n}<r$ and therefore we have $$B\Big(x,\frac{1}{n}\Big)\subseteq B(x,r)\subseteq U.$$

Hence $\beta$ is a local basis at $x$. Hence any $\mathcal{F}$-metric space is first countable.
\end{proof}
Now we know that a set having a countable dense subset is said to be separable. In the following result, we show that the aforementioned property is equivalent to that of second countability, in case of $\mathcal{F}$-metric spaces also.
\begin{theorem}
Every separable $\mathcal{F}$-metric space is second countable.
\end{theorem}
\begin{proof}
Let $(X,d)$ be a separable $\mathcal{F}$-metric space. Then there exists $(f,\alpha)\in \mathcal{F}\times [0,\infty)$ satisfying the conditions $(D1$-$D3)$ of Definition \ref{D1} and a countable subset $A$ of $X$ such that $\bar{A}=X$. Let $\mathcal{B}= \{B(x,r):x \in A,~~r\in \mathbb{Q} ~\mbox{and}~ r>0\}$. 

Now, we show that $ \mathcal{B}$ is a base for the topology $\tau_{\mathcal{F}}$ on $X$ induced by $d$. To show this, let $U$ be any open subset of $X$ and $x \in X$. Then there exists a real number $r>0$ such that $$x\in B(x,r) \subset U.$$
From $(\mathcal{F}_2)$, it follows that for the above $r>0$, there exists $\delta >0$ such that $$ 0< t < \delta \Rightarrow f(t)<f(r) - \alpha.$$ Let $\delta_1 < \delta$ be a positive rational number. Then we have
\begin{align}
f(\delta_1) < f(r) - \alpha. \label{ss1}
\end{align}
Since, $\bar{A}=X$, there exists $a \in A$ such that $a \in B(x, \frac{\delta_1}{2}).$ So, $x \in B(a, \frac{\delta_1}{2}).$ Let, $B=B(a, \frac{\delta_1}{2}).$ Next, we show that $B \subset U$. Let $y \in B$, so $d(a,y) < \frac{\delta_1}{2}$.

Now, if $d(x,y)=0$, then $x=y$ and so $y \in U$ and if $d(x,y) >0$, then we have
\begin{align*}
f(d(x,y)) & \leq f(d(x,a)) +d(a,y)) +\alpha \\
& \leq f \Big(\frac{\delta_1}{2} + \frac{\delta_1}{2}\Big) + \alpha\\
& < f(r) - \alpha + \alpha\\
& =  f(r)\\
\Rightarrow d(x,y) &<r\\
\Rightarrow y \in B(x,r) & \subset U.
\end{align*}
Therefore, $x \in B \subset U$. Also, it is clear that $B \in \mathcal{B}$. Thus, $\mathcal{B}$ is a countable base for the topology $\tau_{\mathcal{F}}$ induced by $d$ and hence $(X,d)$ is second countable.
\end{proof}
It is interesting to note from the above theorems that some of the important topological properties of usual metric spaces also hold in case of these metric spaces. Moreover, metrizability is always an immensely worthwhile aspect for a metric space to own. Therefore, to complete the investigation on topological properties of $\mathcal{F}$-metric spaces, we pose the following interesting open problem.
\begin{Open Prob}
Is an $\mathcal{F}$-metric space metrizable? If not, further, which conditions will guarantee that it is metrizable?
\end{Open Prob}
\section{\bf Fixed point results via altering distance functions}
In this section, we present a few fixed point results concerning some special kinds of self-maps via altering distance functions. To begin with, we recall a crucial notion of an altering distance function which was originally coined by Khan et al. \cite{KSS}.
\begin{definition}
A function $\varphi:[0,\infty)\rightarrow [0,\infty)$ is called an altering distance function if 
\begin{enumerate}
\item[(i)]$\varphi$ is continuous,
\item[(ii)]$\varphi$ is non-decreasing,
\item[(iii)]$\varphi(t)=0\Longleftrightarrow t=0$.
\end{enumerate} 
\end{definition}
Again, in 1962, Edelstein \cite{E1} obtained the following version of the Banach contraction principle \cite{B1} relevant to the contractive mappings.
\begin{theorem}\cite{E1} \label{Edel}
Let $(X,d)$ be a compact metric space and $T:X\rightarrow X$ be a self-map. Assume that $$d(Tx,Ty)<d(x,y)$$ holds for all $x,y \in X$ with $x\neq y$. Then $T$ has a unique fixed point in $ X $.
\end{theorem}
Now, employing the idea of altering distance functions, we generalize Theorem \ref{Edel} in $\mathcal{F}$-metric  setting. The following result assures the existence and uniqueness of a fixed point arising of contractive-type mappings in this framework.
\begin{theorem} \label{t2}
Let $(X,d)$ be a sequentially compact $\mathcal{F}$-metric space and $T$ be a self-map on $X$ such that $$\varphi(d(Tx,Ty))<\varphi(d(x,y))$$ for all $x,y\in X$ with $x\neq y$, where $\varphi$ is an altering distance function. Then $T$ has a unique fixed point and for any $x\in X$, the sequence $\{T^n(x)\}$ is $\mathcal{F}$-convergent to that fixed point.
\end{theorem}
\begin{proof}
Let $x_0 \in X$ and we define a sequence $\{ x_n\}$ by $x_n=T^n(x_0)$ for  $n\in \mathbb{N}$. If $x_n=x_{n+1}$ for some $n\in \mathbb{N}$, then $T$ must have a fixed point $x_n \in X$. So, we assume $x_n\neq x_{n+1}$ for all $n\in \mathbb{N}$. Since $(X,d)$ is a sequentially compact $\mathcal{F}$-metric space, then there exists a convergent subsequence $\{x_{n_{k}}\}$ of $\{x_n\}$ that converges to $z$. Since $T$ is continuous, it follows that the subsequence $\{x_{n_{k}}\}$  converges to $Tz$. Take $s_n=\varphi(d(x_n,x_{n+1}))$ for all $n \in \mathbb{N}$. Then we have,
\begin{align*}
s_{n+1} &= \varphi(d(x_{n+1},x_{n+2}))\\
&< \varphi(d(x_n,x_{n+1}))\\
&= s_n
\end{align*}
This shows that the sequence of non-negative real numbers  $\{s_n\}$ is a decreasing sequence and hence convergent to some $a \geq 0$. Next, if $a>0$, then we have $$0<a = \displaystyle \lim_{k\to \infty} \varphi(d(x_{n_k},x_{n_k+1}))=\varphi(d(z,Tz)).$$
Also we have $$0<a = \displaystyle \lim_{k\to \infty} \varphi(d(x_{n_k+1},x_{n_k+2}))=\varphi(d(Tz,T^2z))<\varphi(d(z,Tz))= a,$$
which leads to a contradiction. So we must have $a=0.$ Thus the sequence $\{s_n\}$  converges to zero. Therefore we obtain
\begin{align*}
\lim_{n\to \infty} \varphi(d(x_n,x_{n+1}))&=0\\
 \implies \lim_{n\to \infty} d(x_n,x_{n+1})&=0.
\end{align*}
Now we show that $z$ is a fixed point of $T$. On the contrary, let $z\neq Tz$. Then $d(z,Tz)>0$ and so by $(D3)$ we have $$f(d(z,Tz)) \leq f\{d(z,x_{n_k})+d(x_{n_k},x_{n_k+1})+d(x_{n_k+1},Tz)\} \to -\infty ~\mbox{as} ~k \to \infty,$$ which is a contradiction. So we must have $z=Tz$, i.e., $z$ is a fixed point of $T$.

Now, to prove the uniqueness of the fixed point, if possible assume, $T$ has two fixed points $z_1,z_2\in X$ with $z_1\neq z_2$, i.e., $Tz_1=z_1~~\mbox{and}~~ Tz_2=z_2$. Then 
\begin{align*}
\varphi(d(z_1,z_2))&=\varphi(d(Tz_1,Tz_2))\\
&<\varphi(d(z_1,z_2)),
\end{align*}
which is impossible. Therefore, we obtain $z_1=z_2$. Hence $T$ has a unique fixed point. Eventually, we prove that sequence $\{x_n\}$ converges to $z$. If $x_n=z$ for finitely many $n\in \mathbb{N}$, then we can exclude those $x_n$ from $\{x_n\}$ and assume that, $x_n\neq z$ for all $n\in \mathbb{N}$. Then from the sequentially compactness, we have $$\lim_{k\to \infty}d(x_{n_k},z)=0,$$ that is, $z$ is the accumulation point of the sequence $\{x_n\}$. Again, if $z_1$ be another accumulation point of $\{x_n\}$, then there exists a subsequence of $\{x_n\}$ which converges to $z_1$. Then continuing as above discussion, we can show that $z_1$ is a fixed point of $T$, which implies that $z=z_1$. So, $z$ is the unique accumulation point of $\{x_n\}$.

Now, we consider a sequence $\{\alpha_n\}$ of real numbers with $\alpha_n=\varphi(d(x_n,z))$ for all $n\in \mathbb{N}$. Therefore, $\alpha_{n_k}=\varphi(d(x_{n_k},z))\rightarrow 0$ as $n\rightarrow \infty$, which implies $\{\alpha_n\}$ has a subsequence $\{\alpha_{n_k}\}$ that converges to $0$. So $0$ is an accumulation point of $\{\alpha_n\}$. Now, we have, 
\begin{align*}
\alpha_{n+1}=&\varphi(d(x_{n+1},z))\\
=& \varphi(d(Tx_n,Tz))\\
<& \varphi (d(x_n,z))\\
=& \alpha_n
\end{align*}
for all $n\in \mathbb{N}$. Hence $\{\alpha_n\}$ is monotonic decreasing sequence of non-negative real numbers and $0$ is an accumulation point of $\{\alpha_n\}$. Then $\{\alpha_n\}$ must converge to $0$. Therefore, letting $n\rightarrow \infty$, we obtain $\varphi(d(x_n,z))\rightarrow 0$. This implies
$$\lim_{n\rightarrow \infty}d(x_n,z)=0.$$
Hence $\{x_n\}$ converges to $z$ and so $\{T^nx_0\}$ converges to $z$. Since $x_0\in X$ is arbitrary, $\{T^nx\}$ converges to $z$ for each $x\in X$.
\end{proof}
From the above theorem, we can establish the subsequent corollary by taking $\varphi(t)=t$ for all $t\in [0,\infty)$.
\begin{corollary}
Let $(X,d)$ be a sequentially compact $\mathcal{F}$-metric space and $T:X\rightarrow X$ be a mapping such that $$d(Tx,Ty)< d(x,y)$$ for all $x,y\in X$ with $x\neq y$, where $\varphi$ is an altering distance function.  Then $T$ has a unique fixed point in $X$ and for any $x\in X$, the sequence $\{T^nx\}$ converges to that fixed point.
\end{corollary}
The succeeding example authenticates previously discussed Theorem \ref{t2}.
\begin{example}
Let $X=[0,1]$ and define the metric $d:X\times X\rightarrow [0,\infty)$ by $$d(x,y)=|x-y|$$ for all $x,y\in X$. Also consider a self-map $T$ on $X$ by $$T(x)=1-\frac{x}{2},$$ for all $x\in X$. Then $(X,d)$ is a $\mathcal{F}$-metric space with $f(t)=\ln t$ and $\alpha =0$ and also $\mathcal{F}$-compact. Now, consider $\varphi(t)=t^2,~~ t\in [0,\infty)$. We have,
\begin{align*}
\varphi(d(Tx,Ty))&=\varphi(|Tx-Ty|)\\
&=|Tx- Ty|^2\\
&=\Big(1-\frac{x}{2}-1+\frac{y}{2}\Big)^2\\
&=\frac{1}{4}(x-y)^2\\
& <(x-y)^2\\
& =\varphi(d(x,y)),
\end{align*}
for all $x,y\in X$ with $x\neq y.$ Therefore, $$\varphi(d(Tx,Ty))<\varphi(d(x,y))$$ holds for all $x,y\in X$ with $x\neq y$. Thus $T$ satisfies all the hypotheses of Theorem \ref{t2} and hence possesses a unique fixed point $x=\frac{2}{3}\in X$.
\end{example}
In the next theorem, we consider Kannan-type contractive mappings defined on an $\mathcal{F}$-metric space.
\begin{theorem} \label{t3}
Let $T$ be a self-mapping on an $\mathcal{F}$-metric space $(X,d)$. Suppose there exists $x_0\in X$ such that the orbit $\phi(x_0)=\{T^n(x_0):n\in \mathbb{N}\}$ has an accumulation point $z\in X$. If $T$ is orbitally  continuous at $z$ and there exists an altering distance function $\varphi$ such that 
$$\varphi(d(Tx,Ty))<\frac{1}{2}\{\varphi(d(x,Tx))+\varphi(d(y,Ty))\}$$ holds for all $x,y=Tx\in \phi\overline{(x_0)}$ with $x\neq y$, then $z$ is the unique fixed point of $T$.
\end{theorem}
\begin{proof}
Let us define a sequence $\{x_n\}$ by $Tx_n=x_{n+1}$ for $n\in \mathbb{N}_0$. If $x_n=x_{n+1}$ for some $n$, then $T$ must have a fixed point. Now, we assume $x_n\neq x_{n+1}$ for every $n\in \mathbb{N}_0$. If $\alpha_n=\varphi(d(x_n,x_{n+1}))$, then by the given condition it follows that
\begin{align*}
\alpha_{n+1}&=\varphi(d(x_{n+1},x_{n+2}))\\
&=\varphi(d(Tx_n,Tx_{n+1}))\\
&<\frac{1}{2}\{\varphi(d(x_n,Tx_n))+\varphi(d(x_{n+1},Tx_{n+1}))\}\\
&=\frac{1}{2}\{\varphi(d(x_n,x_{n+1}))+\varphi(d(x_{n+1},x_{n+2}))\}\\
&=\frac{1}{2}\{\alpha_n+\alpha_{n+1}\}\\
&<\alpha_n\\
\Rightarrow \alpha_{n+1} &<\alpha_n.
\end{align*}
This shows that $\{\alpha_n\}$ is a strictly decreasing sequence of positive reals and it must be convergent to some non-negative real number $a$. Since $\phi(x_0)$ has an accumulation point, there exists a sequence of positive integers $\{n_k\}$ such that $\{x_{n_k}\}$ converges to some $z\in X$. Therefore, by orbitally continuity of $T$, we get $\{x_{n_k+1}\}$ converges to $Tz$.

Now, we claim that $a=0$. If $a \neq 0$, then we have $$0<a =\displaystyle \lim_{k\to \infty} \varphi(d(x_{n_k},x_{n_k+1}))=\varphi(d(z,Tz)).$$
Also we have $$0<a = \displaystyle \lim_{k\to \infty} \varphi(d(x_{n_k+1},x_{n_k+2}))=\varphi(d(Tz,T^2z))<\varphi(d(z,Tz)) = a,$$
which leads to a contradiction. So we must have $a=0.$ Thus the sequence $\{\alpha_n\}$  converges to zero. Let there exists $(f,\alpha)\in \mathcal{F}\times [0,\infty)$ satisfying the conditions $(D1$-$D3)$ of Definition \ref{D1}. Then by $(\mathcal{F}_2)$, for a given $\epsilon>0$, there exists a $\delta>0$ such that 
\begin{align} \label{t3eq1}
0<t<\delta\ \ \Rightarrow f(t)<f(\varphi(\epsilon))-\alpha.
\end{align}
Now, \begin{align*}
\varphi(d(x_n,x_{n+1}))&<\frac{1}{2}\{\varphi(d(x_{n-1},Tx_{n-1}))+\varphi(d(x_n,Tx_n))\}\\
&=\frac{1}{2}\{\varphi(d(x_{n-1},x_n))+\varphi(d(x_n,x_{n+1}))\}\\
&=\frac{1}{2}\{\varphi(d(T^{n-1}x_0,T^nx_0))+\varphi(d(T^nx_0,T^{n+1}x_0))\}.
\end{align*}
Finally we obtain, $$\sum_{i=n}^{m-1}\varphi(d(x_i,x_{i+1}))<\sum_{i=n}^{m-1}\frac{1}{2}\{\varphi(d(T^{i-1}x_0,T^ix_0))+\varphi(d(T^ix_0,T^{i+1}x_0))\}.$$
Since $$\lim_{n\rightarrow \infty}\{\varphi(d(T^{n-1}x_0,T^nx_0))+\varphi(d(T^nx_0,T^{n+1}x_0))\}=0,$$ 
there exists some $N\in \mathbb{N}$ such that $$0<\sum_{i=n}^{m-1}\varphi(d(x_i,x_{i+1}))<\delta,$$
holds for all $n\geq N$. Hence by \eqref{t3eq1} and ($\mathcal{F}_1$), we have
\begin{align} \label{t3eq2}
 f\Big(\sum_{i=n}^{m-1}\varphi(d(x_i,x_{i+1}))\Big)< f(\varphi(\epsilon))-\alpha.
\end{align}
Now, we show that $$d(x_n,x_m) < \epsilon$$ for all $m>n \geq N$. Let $m,n \in \mathbb{N}$ be fixed but arbitrary such that $m>n \geq N$. If $d(x_n,x_m) = 0$, then clearly $d(x_n,x_m) < \epsilon$ and if $d(x_n,x_m) > 0$, then using (D3) and \eqref{t3eq2}, we get 
\begin{align*}
 \varphi(d(x_n,x_m))&>0\\
\Rightarrow f(\varphi(d(x_n,x_m)))&\leq f\Big(\sum_{i=n}^{m-1}\varphi(d(x_i,x_{i+1}))\Big)+\alpha\\
&<f(\varphi(\epsilon)),
\end{align*}
which gives by $(\mathcal{F}_1)$ that
\begin{align*}
\varphi(d(x_n,x_m))<& \varphi(\epsilon)\\
\Rightarrow d(x_n,x_m)<& \epsilon,
\end{align*}
for all $m>n\geq N.$
This proves that $\{x_n\}$ is $\mathcal{F}$-Cauchy. Since $\{x_n\}$ has a subsequence $\{x_{n_k}\}$ which converges to $z$, the limit of $\{x_n\}$ will be $z$. This means that
\begin{equation} \label{t3eq3}
\lim_{n\rightarrow \infty}d(x_n,z)=0.
\end{equation}
This implies that $\{x_n\}$ tends to $z$ and by orbitally continuity of $T$, we get $\{Tx_n\}$ tends to $Tz$. Since $Tx_n=x_{n+1}$, we have, by the uniqueness of limit of sequence, $z=Tz.$
Hence $z$ is a fixed point of $T$. For uniqueness, let $z^*$ be another fixed point of $T$. Then
\begin{align*}
\varphi(d(z,z^*))&=\varphi(d(Tz,Tz^*))\\
&<\frac{1}{2}\{\varphi(d(z,Tz))+\varphi(d(z^*,Tz^*))\}\\
&<0,
\end{align*}
a contradiction. Therefore $z$ is the unique fixed point of $T$.
\end{proof}
\begin{remark}
The following example shows that, if we take any two points $x,y=Tx\in \phi(x_0)$ with $x\neq y$ satisfying the inequality
$$\varphi(d(Tx,Ty))<\frac{1}{2}\{\varphi(d(x,Tx))+\varphi(d(y,Ty))\},$$ then the sequence of iterates $\{T^n(x_0)\}$ may not converge to the accumulation point of $\phi(x_0)$.
\end{remark}
\begin{example}
Let $X=\{2,-2,2+\frac{1}{3n},-2-\frac{1}{3n+1}:n\in \mathbb{N}\}$ and define the metric $d:X\times X\rightarrow [0,\infty)$ by $$d(x_1,y_1)=|x_1-y_1|$$ for all $x_1,y_1\in X$. Then $(X,d)$ is a $\mathcal{F}$-metric space with $f(t)=\ln t$, for $t>0$ and $\alpha=0$. Now, we define $T$ on $X$ by 
\begin{align*}
T(2)=-2,~~& T(-2)=2,~~T\Big(2+\frac{1}{3n}\Big)=-2-\frac{1}{3n+1}\\~~\mbox{and}~~ & T\Big(-2-\frac{1}{3n+1}\Big)=2+\frac{1}{3(n+1)}.
\end{align*}
 Now, for $x_0=2+\frac{1}{3}$, we have $$\phi(x_0)=\Big\{2+\frac{1}{3},-2-\frac{1}{4},2+\frac{1}{6},-2-\frac{1}{7},2+\frac{1}{9},-2-\frac{1}{10},\cdots\Big\}.$$
Then it can be easily verified that the inequality $$\varphi(d(Tx,Ty))<\frac{1}{2}\{\varphi(d(x,Tx))+\varphi(d(y,Ty))\}$$ is satisfied for all $x,y=Tx\in \phi(x_0)$ with $x\neq y$ and $\varphi(t)=t,\; t\geq 0$. Whenever $z=2$ in $\phi\overline{(x_0)}$, we have $Tz=-2$. Moreover, $T$ is orbitally continuous at $z$. But still $z$ is not a fixed point of $T$. 
\end{example}
\begin{example}
Let $\{e_i\}$ be the sequence of real numbers whose $i$-th term is $1$ and all other terms are $0$. Take $X=\{e_i: i\geq 1\}$. Define $d: X \times X \to \mathbb{R}$ by 
\[
d(x,y) = \left\{\begin{array}{ll}
        1+\left|\frac{1}{\sup_i |e_i| }-\frac{1}{\sup_i |e_j|}\right|, & \text{if} x=e_i~\mbox{and}~ y=e_j~\mbox{with}~ x\neq y;\\
       0, & \text{if } x=y.
       \end{array}\right.     
\]
Then it is easy to note that $(X,d)$ is an $\mathcal{F}$-metric space with $f(t)=\ln t$ and $\alpha=0$. Also, it is clear that $(X,d)$ is $\mathcal{F}$-complete but not $\mathcal{F}$-compact.

Now, define a mapping $T: X \to X$ by $$T(e_i)=e_{3i}$$ for all $e_i \in X$. Therefore, for $i < j$, we have 
\begin{align*}
d(Te_i,Te_j)&= 1+\left|\frac{1}{3i}-\frac{1}{3j}\right|\\
&= 1+\frac{1}{3i}-\frac{1}{3j}< 1+\frac{1}{3i};
\end{align*}
whereas, 
\begin{align*}
\frac{1}{2}\{d(e_i,Te_i)+d(e_j,Te_j)\}&=\frac{1}{2}\left\{1+\left|\frac{1}{i}-\frac{1}{3i}\right|+1+\left|\frac{1}{j}-\frac{1}{3j}\right|\right\}\\
&=1+\frac{1}{3i}+\frac{1}{3j}>1+\frac{1}{3i}.
\end{align*}
So, $$d(Te_i,Te_j)<\frac{1}{2}\{d(e_i,Te_i)+d(e_j,Te_j)\}.$$ In a similar manner, we can show that $$d(Te_i,Te_j)<\frac{1}{2}\{d(e_i,Te_i)+d(e_j,Te_j)\}$$ if $i>j$. Therefore, $$d(Tx,Ty)< \frac{1}{2}\{d(x,Tx) + d(y,Ty)\}$$ for all $x, y \in X$ with $x \neq y$, but $T$ has no fixed point.
\end{example}
From the above example, we see that the completeness of $X$ can not alone guarantee the existence of a  fixed point for the Kannan-type contractive mappings.
\begin{theorem}
Let $(X,d)$ be an $\mathcal{F}$-complete metric space and $T$ be a continuous self-map on $T$ such that 
\begin{equation} \label{k1}
\varphi(d(Tx,Ty))< \frac{1}{2}\{\varphi(d(x,Tx)) +\varphi(d(y,Ty))\}
\end{equation}
 for all $x,y \in X$ with $x\neq y$, where $\varphi$ is an altering distance function. Also assume that
 for any $x\in X$ and for any $\epsilon>0$, there exists $\delta>0$ such that $$\varphi(d(T^ix,T^jx))<\epsilon + \delta \Rightarrow \varphi(d(T^{i+1}x,T^{j+1}x))\leq\epsilon$$  for any $i,j\in \mathbb{N}$. Then $T$ has a unique fixed point, and for any $x \in X$, the sequence of iterates $\{T^nx\}$ converges to that fixed point.
\end{theorem}
\begin{proof}
Let $x_0\in X$ be arbitrary but fixed. We consider the sequence $\{x_n\}$ in $X$, where $x_n = T^n x_0$ for all natural numbers $n$. Also, we take the sequence of real numbers $\{s_n\}$ defined by $s_n=\varphi(d(x_n,x_{n+1}))$ for all $n \in \mathbb{N}$.

If $x_n=x_{n+1}$ for some $n$, then it is easily noticeable that $x_n$ is a fixed point of $T$. So now we assume that $x_n\neq x_{n+1}$ for all $n \in \mathbb{N}$. Putting $x=x_n,~y=x_{n+1}$ in \eqref{k1} we get,
\begin{align*}
\varphi(d(Tx_n,Tx_{n+1}))&<\frac{1}{2}\{\varphi(d(x_{n},Tx_{n}))+\varphi(d(x_{n+1},Tx_{n+1}))\}\\
\Rightarrow \varphi(d(x_{n+1},x_{n+2}))&<\frac{1}{2}\{\varphi(d(x_{n},x_{n+1}))+\varphi(d(x_{n+1},x_{n+2}))\}\\
\Rightarrow s_{n+1} &< \frac{1}{2}\{s_n + s_{n+1}\}\\
\Rightarrow s_{n+1} &< s_n.
\end{align*}
Therefore, $\{s_n\}$ is a decreasing sequence of non-negative real numbers and hence convergent to some $a\geq 0$. We claim that $a=0$. If $a>0$, then by the given condition there exists $\delta>0$ such that 
\begin{equation} \label{k2}
\varphi(d(T^ix,T^jx))<a + \delta \Rightarrow \varphi(d(T^{i+1}x,T^{j+1}x))\leq a 
\end{equation}
  for any $i,j\in \mathbb{N}$. But since $\{s_n\}$ converges to $a$, there exists $n\in \mathbb{N}$ such that $$s_n < a+ \delta.$$ Then using \eqref{k2} we get, $$s_{n+1} \leq a,$$ which contradicts the fact that $\{s_n\}$ converges to $a$. Therefore, $\{s_n\}$ converges to $0$.
  
  Now, we show that $\{x_n\}$ is a Cauchy sequence. To show this, we put $x=x_n,y=x_m$ in \eqref{k1} and get,
\begin{align*}
\varphi(d(Tx_n,Tx_m))&<  \frac{1}{2}\{\varphi(d(x_n,Tx_n))+\varphi(d(x_m,Tx_m))\}\\
\Rightarrow \varphi(d(x_{n+1},x_{m+1}))&<  \frac{1}{2}\{s_n +s_m\} \to 0 ~~ \mbox{when}~~ n,m \to \infty. 
\end{align*}
Therefore, the double sequence $ \{\varphi(d(x_{n},x_{m}))\}$  of real numbers converges to $0$. So for any $\epsilon >0$, there exists $N \in \mathbb{N}$ such that 
$$\varphi(d(x_{n},x_{m})) < \varphi(\epsilon)$$ for all $n,m \geq N$, which gives
$$d(x_{n},x_{m}) < \epsilon$$ for all $n,m \geq N$. Therefore $\{x_n\}$ is a Cauchy sequence in $X$. Again, as $X$ is $\mathcal{F}$-complete,  $\{x_n\}$  converges to some $z \in X$. Since, $T$ is continuous, $\{Tx_n\}$ converges to $Tz$, i.e., $\{x_n\}$  converges to $Tz$. So we have $z=Tz$, i.e., $z$ is a fixed point $T$. The uniqueness of the fixed point can be similarly derived as in Theorem \ref{t3}.
\end{proof}
%
\begin{Acknowledgement}
{\noindent This research is funded by Council of Scientific and Industrial Research (CSIR), Government of India under Grant Number: $25(0285)/18/EMR-II$.}
\end{Acknowledgement}
\bibliographystyle{plain}

\end{document}